\documentclass[12pt]{amsart}
\usepackage{amssymb,amsmath,url}

\theoremstyle{plain}
\newtheorem{theorem}{Theorem}
\newtheorem{lemma}[theorem]{Lemma}
\newtheorem{proposition}[theorem]{Proposition}

\newtheorem{fact}[theorem]{Fact}
\newtheorem{problem}{Problem}

\theoremstyle{definition}

\DeclareMathOperator{\im}{{\sf im}}

\DeclareMathOperator{\End}{{\sf End}}
\DeclareMathOperator{\LEnd}{{\sf End}^\ast}
\DeclareMathOperator{\id}{{\sf id}}

\newcommand{\sk}[2]{\langle{#1}\mid{#2}\rangle}

\newcommand{\lab}[1]{\label{#1}}

\begin{document}

\title[Direct finiteness]{Direct finiteness
of representable regular $*$-rings}

\author[C. Herrmann]{Christian Herrmann}
\address{Technische Universit\"{a}t Darmstadt FB4\\ Schlo{\ss}gartenstr. 7\\ 64289 Darmstadt\\ Germany}
\email{herrmann@mathematik.tu-darmstadt.de}

\dedicatory{Dedicated to the memory of Susan M. Roddy}


\maketitle

\begin{abstract}
For $*$-regular rings $R$  of endomorphisms of inner product spaces
we provide a condition on $r$ sufficient for $r$ being a unit.
In contrast to the claim made in earlier
versions of this note it remains open whether this condition can be used
to establish direct finiteness of $R$. 
\end{abstract}

\setcounter{section}{-1}
\section{Erratum}
A long standing problem of Handelman's  (cf. Problem 48 in \cite{good}) asks whether 
all $*$-regular rings are directly finite.
In \cite[Thm.3.1]{arch}  and, in more general context, in \cite[Thm.4]{arch2} direct finiteness
 has been claimed for representable $*$-regular
rings, i.e. for regular subrings $R$
of endomorphisms   of inner product spaces $(V_F,\Phi)$ where  a proper involution
on $R$ is given by 
$r \mapsto r^*$, the adjoint of $r$.

Though, what has been shown is only the following:
Let 
 $\mathcal{F}$ denote  the set of  finite dimensional 
subspaces $U$ of $V_F$.
 Then $r$ is a unit of $R$ 
provided that $(*)$: For each $U\in \mathcal{F}$
 there are $W\supseteq U$ in $\mathcal{F}$ 
and  a unit $\alpha $ of ${\sf End}(W_F)$ such that
$\alpha|U=r|U$ and $\alpha^*|U=r^*|U$.

In \cite{arch} it has been tacitly assumed that 
$(*)$ holds if $sr=1$ for some $s\in R$. No proof has been given.
 Thus, the following claims remain without proof: 
Thm. 3.1 in \cite{arch}, Thm. 4  and Cor. 6 in \cite{arch2},
Fact 5 and Thm. 7 in  \cite{he3}. Lemma 5 in \cite{arch2}
is wrong:  $(*)$ fails  for the   shift matrix $r$ 
in the $*$-ring  of row and column finite matrices over $\mathbb{R}$ 
(which  is  not regular, 
as this author learned from
Fred Wehrung \cite[Ex.8.19]{fred}, cf. \cite{count}) 
embedded (as ring) into ${\sf End}(V_\mathbb{R})$,
$\dim V_\mathbb{R} =\aleph_0$.

To the best of this authors knowledge, the problem
of direct finiteness  remains open
for the class of  $*$-regular rings as well as for
its subclass of representables and, more specially,
  of semi-artinians.

\section{Introduction}
A $*$-ring, that is ring with involution,
is called \emph{finite} if  $rr^*=1$  implies $r^*r=1$. This is a basic notion
in the classification of von Neumann algebras; in particular,
as shown by Murray and von Neumann,
a finite von Neumann algebra admits a finite
 $*$-ring of
quotients.  This ring is also $*$-regular
and \emph{directly finite}, that is $rs=1$ implies $sr=1$.
As Ara and Menal \cite{ara} have shown,
any  $*$-regular ring is at least finite, while
direct finiteness remains an open question,
as stated by 
 Handelman \cite[Problem 48]{good}.
The present note gives for $*$-regular rings $R$ of endomorphisms 
of inner product spaces a condition in terms of the  action of $r$ on finite
dimensional subspaces 
sufficient  for $r$ being a  unit.

For $\ast$-rings, there is a natural and well established
concept of [faithful] representation in a vector space $V_F$
endowed with a non-degenerate orthosymmetric sesquilinear form:
an embedding  into the $*$-ring $\LEnd(V_F)$
of those endomorphisms of $V_F$ which admit an adjoint.
Famous examples 
are due to Gel'fand-Naimark-Segal ($C^\ast$-algebras in Hilbert space) and
Kaplansky (primitive $\ast$-rings with a minimal right ideal).
For $\ast$-regular rings of classical quotients
of finite Rickart $C^\ast$-algebras
existence of  representations has been established in \cite{PartI},
jointly with M.~Semenova.
N.~Niemann \cite{Nik,nik2}
has shown that a subdirectly irreducible  $*$-regular ring 
is representable
if and only  its ortholattice lattice 
of principal right ideals is representable within the ortholattice
of closed subspaces of some $V_F$.

According to joint work with Susan M. Roddy \cite{hr2},
representability of modular ortholattices 
 is equivalent to 
membership in a variety 
generated by finite height members.
Using ideas from Tyukavkin \cite{tyu},
the analogue  for  $*$-regular rings
was obtained by F.~Micol \cite{Flo}. Here, we rely
on the presentation given in \cite{act}:
A regular $*$-ring can be represented within $V_F$, respectively
some ultrapower thereof,
if and only if it can be obtained
 via formation of ultraproducts,
regular $*$-subrings, and homomorphic images from 
the class of the $\LEnd(U_F)$, $U_F$
ranging over finite dimensional
 non-degenerate subspaces of $V_F$.
In the setting of  \cite{act}
we derive a condition on $r$ sufficient for  $r$ 
being a unit in $R$. It remains open whether this can be used
to establish direct finiteness.

Thanks are due to Ken Goodearl for the hint to
reference \cite{FU}.

\section{Preliminaries}
When mentioning rings, we always mean associative rings $R$ with  unit $1_R$,
considered as a constant.
In any ring, if $a$ has a \emph{left inverse} $x$, that is $xa=1$,
then $a$ is \emph{left cancellable}, that is $ay=az$ implies $y=z$.
$R$ is \emph{directly finite} if, for all $r,s \in R$, $sr=1$ implies $rs=1$.
In such ring, $a$ has a left inverse  if and only if $a$ is a unit (and $x=a^{-1}$).
The endomorphism ring $\End(V_F)$ of a vector space is
directly finite if and only if $\dim V_F<\omega$.
A $\ast$-\emph{ring} is a ring endowed with an involution $r \mapsto r^*$; 
an element $e$ of such a ring  is a \emph{projection}, if $e=e^2=e^\ast$.

 A ring $R$ is [\emph{von Neumann}] \emph{regular} if for any $a\in R$,
there is an element $x\in R$ such that $axa=a$; such an element is called a \emph{quasi-inverse} of $a$. 
If, for all $a$,
 $x$ can be chosen a unit, then $R$ is \emph{unit regular}.
 Examples of such are the $\End(V_F)$, $\dim V_F<\omega$.
A detailed discussion of direct finiteness in 
 regular rings is given in Goodearl \cite{good}.
A regular $*$-ring is $*$-\emph{regular} if $xx^*=0$ only for $x=0$.

In a regular ring, any left cancellable $a$
has a left  inverse; indeed $axa=a$ implies 
$xa=1$. 
If $xa=1$ and  $aua=a$ with a unit $u$ then $a=u^{-1}$ and $x=a^{-1}$.
 It follows
\begin{fact}\lab{un} In a directly finite regular ring
every left cancellable element is a unit --
similarly on the right. Every
unit regular ring is directly finite. 
\end{fact}
\begin{fact}\label{unit}
If $R$ is a  regular subring of the ring $R`$ and $r \in R$
a unit in $R'$ then $r$ is a unit in  $R$.
\end{fact}

We recall some basic concepts and facts from \cite{act}
(here, $\Lambda$
can be  taken the $\ast$-ring of integers).
In the sequel,  $F$ will be a division ring  endowed with an involution
and $V_F$ a [right] $F$-vector space 
of $\dim V_F>1$ endowed with a
non-degenerate sesquilinear form $\sk{.}{.}$
which is \emph{orthosymmetric}, that is  
$\sk{v}{w}=0$ iff $\sk{w}{v}=0$.
 Such space 
will be called \emph{pre-hermitean} and denoted by $V_F$, too.
Within such space, any endomorphism $\varphi$ has
at most one adjoint $\varphi^*$;
and these $\varphi$ form a subring of $\End(V_F)$
which is a $\ast$-ring $\LEnd(V_F)$ under the involution $\varphi\mapsto
 \varphi^*$. For $\dim V_F<\omega$, $\LEnd(V_F)$
contains all of $\End(V_F)$.
A [faithful] \emph{representation} of a $\ast$-ring $R$ is an
embedding of $R$ into some $\LEnd(V_F)$.

\medskip
Consider a  linear subspace $U$ of $V_F$, $1 <\dim U_F<\omega$. 
With 
 the induced sesquilinear form,
 $U_F$ is pre-hermitean  if and only if $V=U\oplus U^\perp$;
in particular, there is a projection $\pi_U \in \LEnd(V_F)$
such that $U=\im \pi_U$ and such that
the inclusion map 
$\varepsilon_U:U \to V$ is the adjoint of $\pi_U$
(here, considered as a map $V \to U$).
We write in this case $U\in\mathbb{O}(V_F)$ and say that
$U$ is a \emph{finite-dimensional orthogonal summand}.
A crucial fact is that $V_F$ is the directed union of the
$U_F$, $U\in \mathbb{O}(V_F)$.
Let $C(F)$ denote the center of $F$ and, for $U\in\mathbb{O}(V_F)$,
\[\begin{array}{lcl}
B_U&=&\{\varepsilon_U\varphi\pi_U+\lambda\id_V\mid\varphi\in\LEnd(U_F),\ \lambda\in 
C(F)\}\\
&=&\{\psi \in \End(V_F)\mid \psi(U) \subseteq U \;\&\;
\exists \lambda \in C(F)\,\psi|U^\perp =\lambda \id_{U^\perp}\}
\end{array}\]
($\varphi$ and $\psi$ are related via $\varphi(v)= \psi(v)+\lambda v$).
Thus,
$B_U$ is a  $\ast$-subring of
$\LEnd(V_F)$ and embeds into $\LEnd(W_F)$
for any $W\in \mathbb{O}(V_F)$, $U \subset W \neq U$.
In particular, $B_U$ is directly finite. Moreover, $B_U$ is unit regular; indeed,
$\chi\in B_U$ is a unit  quasi-inverse of $\psi\in B_U$ 
if $\chi|U$ is one of $\psi|U$ and $\chi|U^\perp= (\psi|U^\perp)^{-1}$
(considering these as endomorphisms of $U$ and $U^\perp$, respectively). 
We put 
\[\begin{array}{lcl}
J(V_F)&=&\{\varphi\in\LEnd(V_F)\mid\dim\im\varphi<\omega\}\\
\hat{J}(V_F)&=&\{\varphi+\lambda \id_V\mid \varphi\in J(V_F),\,\lambda \in C(F)
\}
\end{array}\]
According to  \cite[Proposition 4.4]{act}
$\hat{J}(V_F)$ is a $\ast$-subring
of $\LEnd(V_F)$ and
 $J(V_F)$ is  an ideal of $\LEnd(V_F)$ closed under the involution.
Also,
the following holds. 
\begin{itemize}
\item[($*$)]
For any finite  $\Phi\subseteq J(V_F)$ there is $U\in \mathbb{O}(V_F)$
such that $\varphi=\pi_U\varphi =\varphi \pi_U$ for all $\varphi\in \Phi$.
\end{itemize}
Thus,
$\hat{J}(V_F)$ is  the directed union of the $B_U,\,U\in \mathbb{O}(V_F)$,
whence  unit-regular. 

\begin{lemma}\lab{atex2}
Every regular $\ast$-subring $R$ of $\LEnd(V_F)$
extends to a  regular $\ast$-subring $\hat{R}$ of $\LEnd(V_F)$
containing $\hat{J}(V_F)$ and such that $J(V_F)$ is an ideal of $\hat{R}$.
\end{lemma}
\begin{proof}
$\{ \lambda \varphi\mid \varphi \in R,\,\lambda \in C(F)\}$
is a regular $\ast$-subring $R'$ of 
 $\LEnd(V_F)$
and \cite[Proposition 4.5]{act} applies to $R'$.
\end{proof}

Recall that a [faithful] representation of a $\ast$-ring $R$ within a pre-hermitian space $V_F$ is an embedding  $\varepsilon\colon R\to\LEnd(V_F)$.
It is convenient to consider representations as unitary $R$-$F$-bimodules $_RV_F$ (where the action of $R$ is given as $rv=\varepsilon(r)(v)$)
with sesquilinear form on $V_F$; that is, a $3$-sorted structure
with sorts $V$, $R$, and $F$. Considering a $*$-subring $A$
of $R$ we may add a fourth sort, $A$, and the embedding map.
to obtain $(_RV_F;A)$.
Any elementary extension $(_{\tilde{R}}\tilde{V}_{\tilde{F}};\tilde{A})$
is again such a structure, that is, a representation of 
$\tilde{R}$ and a $*$-ring $\tilde{A}$ which may be considered
as $*$-subring of $\tilde{R}$.
It is a \emph{modestly saturated} extension
if, for each set $\Sigma(\bar x)$ of first order formulas
in finitely many [sorted] variables and with
parameters from  $(_RV_F;A)$, one has
$(_{\tilde{R}}\tilde{V}_{\tilde{F}};\tilde{A})\models \exists \bar x.\Sigma(\bar x)$,
provided that
$\Sigma(\bar x)$ is \emph{finitely realized} in $(_RV_F;A)$, that is
$(_RV_F;A)\models \exists \bar x.\Psi(\bar x)$
 for every finite subset $\Psi(\bar x)$ of $\Sigma(\bar x)$.
Such extension  always exists, cf. \cite[Corollary 4.3.1.4]{CK}.

\section{A sufficient  condition for invertibility}
\begin{proposition}
Let $R$ be a regular $*$-ring represented
within the pre-hermitean space $V_F$. 
Then $r \in R$ is a unit in $R$ provided that the following holds
\begin{itemize}
\item[(**)] For each $U \in \mathbb{O}(V_F)$ there are $W\in \mathbb{O}(V_F)$ 
and  a unit $\alpha $ of ${\sf End}(W_F)$ such that
$\alpha|U=r|U$ and $\alpha^*|U=r^*|U$.
\end{itemize}
\end{proposition}

\begin{proof}
We recall the relevant steps of  the proof of \cite[Theorem 10.1]{act}. 
Given a re\-presentation $_RV_F$
of the regular $*$-ring $R$,
we may assume that $\dim V_F\geqslant\omega$.
In view of Proposition \ref{atex2}
and Fact \ref{unit}, we also may assume
that $R$ is a $*$-subring of $\LEnd(V_F)$ containing $A=\hat{J}(V_F)$
and having ideal  $J(V_F)$. 
Choose $(_{\tilde{R}}\tilde{V}_{\tilde{F}};\tilde{A})$
a modestly saturated elementary extension of $(_RV_F;A)$.

 Let $J_0$ denote the set of projections in $J(V_F)$.
For $a\in\tilde{A}$ and $r\in R$, we put
$a\sim r$ if
$ae=re$ and $a^\ast e=r^\ast e$ for all $e\in J_0$.
According to Claims 1--4 in the proof of 
\cite[Theorem 10.1]{act},
$S=\{a \in \hat{A}\mid a\sim r\ \text{for some}\ r\in R\}$ is a
regular  $\ast$-subring 
of $\hat{A}$ and 
there is a surjective homomorphism $g:S\to R$
such that  $g(a)=r$ if and only if $a\sim r$.
Being an elementary extension of $A$, $\tilde{A}$ 
is  directly finite and so is its subring $S$.

Now, 
  consider a finite set $E\subseteq J_0$.
According to $(*)$,
 there  is $e\in J_0$ such that $ef=f$ 
and $er^\ast f=r^\ast f$
 for all $f\in E$.
Take $a=re$ and observe that $af=ref=rf$ and $a^\ast f =er^\ast f=r^\ast f$ for all $f\in E$ (actually, this follows by the proof of Linells's 
Theorem \cite{FU}).

Thus, by $(**)$ the set
\[
\Sigma(x)=\bigl\{[xe=re]\ \& .\ [x^\ast e=r^\ast e] \& [\exists y.\,yx=1=xy]\mid e\in J_0\bigr\}
\]
 of formulas
with   variables $x,y$ of type $A$
 is finitely realized in $(_{\tilde{R}} \tilde{V}_{\tilde{F}} ;\tilde{A})$.
By saturation,
 there are   $a,b \in \tilde{A}$ 
$ba=1=ba$ and 
 $a\sim r$, whence  $a \in S$ and $g(a)=r$.
By Fact~\ref{unit}, 
$a$ is a unit in $S$.
Moreover, $a$ is left cancellable in $\tilde{R}$ whence in
the  subring $S$.  Hence, $r=g(a)$ 
is a unit of $R$. 
\end{proof}

\begin{problem}
Is $(**)$ also necessary for $r$ being a unit in the represented
$*$-regular ring $R$?

\end{problem}

An example of a simple regular $*$-ring which is not
finite is obtained as follows:
Let $V_F$ a vector space of countably infinite dimension,
and $R=\End(V_F)/J(V_F)$. Of course, $R$ is not directly finite.
 Define the involution
on the direct product $R\times R^{op}$ 
by exchange: $(r,s)^*=(s,r)$ to obtain the $*$-ring $S$.
Now, if $rs=1$ but $sr\neq 1$
 then $xx^*=1$ but $x^*x\neq 1$ in $S$ for $x=(r,s)$.


\begin{thebibliography}{99}

\bibitem{ara}
{Ara, P.}, {Menal, P.}:
On regular rings with involution.
Arch. Math. \textbf{42}, 126--130 (1984)




\bibitem{CK}
{Chang, C.C.}, {H.\,J. Keisler, H.J.}:
Model Theory.
Third ed., Amsterdam (1990)




\bibitem{ehr} 
{Ehrlich, G.}:
Units and one-sided units in regular rings,
Trans. Amer. Math. Soc.
\textbf{216}
(1976), 81--90.





\bibitem{FU}
{C. Faith} and {Y. Utumi}, On a new proof of Litoff's theorem,
 Acta Math.Acad. Sci Hungar. \textbf{14} (1963), 369--371.










\bibitem{good}
{Goodearl, K.R.}:
Von Neumann Regular Rings.
Krieger, Malabar (1991)




\bibitem{nik2} 
Herrmannn, C., Niemann, N.:
On linear representations of $\ast$-regular rings having representable
ortholattice of projections. 
 arXiv:1811.01392 [math.RA]. \url{https://arxiv.org.abs/1811.01392}





\bibitem{hr2}
{Herrmann, C.}, {Roddy, M.S.}:
On varieties of modular ortholattices that are generated by their finite dimensional members.
Algebra Universalis \textbf{72}, 349--357 (2014)



\bibitem{PartI}
{Herrmann, C.}, {Semenova, M.}:
Rings of quotients of finite $AW^\ast$-algebras: Representation and algebraic approximation.
Algebra and Logic \textbf{53}, 298--322 (2014)


\bibitem{act} 
{Herrmann, C.}, {Semenova, M.}:
Linear representations of regular rings and complemented modular
lattices with involution. Acta Sci. Math.
 (Szeged) \textbf{82}, 395--442 (2016)



\bibitem{arch} Herrmann, C.: Direct finiteness of representable regular $·\ast$-rings. Algebra Universalis \textbf{80:1} Paper No. 3, 5 pp. (2019), arXive-math 1901.03555.



\bibitem{arch2}
Herrmann, C.:
Varieties of $\ast$-regular rings.
Math. Slovaca \textbf{70}  no. 4, 815--820 (2020), arXive-math 1904.04505.




\bibitem{he3}
Herrmann, C.:
Unit-regularity and representability for semiartinian $\ast$-regular rings. 
Arch. Math. (Brno) \textbf{56}  no. 1, 43--47 (2020),
arXive-math 1907.13367.


\bibitem{count} Herrmann. C.: Direct finiteness of regular rings with
involution: A counterexample (withdrawn), arXive-math 2408.16437. 







\bibitem{Flo}
{Micol, F.}:
On Representability of $\ast$-Regular Rings and Modular Ortholattices.
PhD thesis, Technische Universit\"{a}t Darmstadt (2003).
\url{http://elib.tu-darmstadt.de/diss/000303/diss.pdf}


\bibitem{Nik}
{Niemann, N.}:
On representability of $\ast$-regular rings in endomorphism rings of vector spaces. PhD thesis
Technische Universit\"{a}t Darmstadt (2007)


\bibitem{tyu}
{Tyukavkin, D.V.}:
Regular rings with involution.
Vestnik Moskovskogo Universiteta. Mate\-matika \textbf{39}, 29--32 (1984)
 (Russian)



\bibitem{fred} 
Wehrung, F.:
Congruences of lattices and ideals of rings. In: Gr\"atzer, G. Wehrung, F. (eds)
 Lattice theory: Special topics and applications. Vol. 1, 297--335, Birkh{\"a}user/Springer, Cham (2014)

\end{thebibliography}
\end{document}